\documentclass[8pt]{article}

\usepackage{lineno,hyperref}
\usepackage{amssymb, amsthm, amsmath}
\usepackage{mathtools}
\usepackage{graphics} 
\usepackage[english]{babel}
\theoremstyle{plain}
\newtheorem{corollary}{Corollary}[section]
\newtheorem{theorem}[corollary]{Theorem}
\newtheorem{lemma}[corollary]{Lemma}

\theoremstyle{definition}
\newtheorem{definition}[corollary]{Definition}

\newtheorem*{observation}{Observation}

\title{The Optimization of Signed Trees}
\author{
  Alvaro Carbonero\footnote{Corresponding Author: ar2carbo@uwaterloo.ca}\\
  \text{University of Waterloo}
  \and
  Janelle Domantay\\
  \text{University of Nevada, Las Vegas}
  \and
  Karen Guthrie\\
  \text{University of Nevada, Las Vegas}
}

\date{\today}

\begin{document}

\maketitle

\begin{abstract}
A signed graph $G$ is a graph where each edge is assigned a + (positive edge) or a - (negative edge). The signed degree of a vertex $v$ in a signed graph, denoted by $sdeg(v)$, is the number of positive edges incident to $v$ subtracted by the number of negative edges incident to $v$. Finally, we say $G$ realizes the set $D$ if:
$$
D = \{sdeg(v) \text{ : } v\in V(G) \}.
$$
The topic of signed degree sets and signed degree sequences has been studied from many directions. In this paper, we study properties needed for signed trees to have a given signed degree set. We start by proving that $D$ is the signed degree set of a tree if and only if $1\in D$ or $-1\in D$. Further, for every valid set $D$, we find the smallest diameter that a tree must have to realize $D$.  Lastly, for valid sets $D$ with nonnegative numbers, we find the smallest order that a tree must have to realize $D$.
\end{abstract}



\section{Introduction}
Every graph in this paper is finite, simple and undirected.

A signed graph is an ordered pair $(G, s)$ where $G$ is a graph and $s: E(G) \rightarrow \{-, + \}$ is a function that assigns either a positive (+) or a negative (-) sign to every edge. Further, we refer to an edge $e$ as positive or negative if $s(e) = +$ or $s(e) = -$ respectively. The notion of signed graphs was introduced by Harary \cite{harary}, and since then they have been widely studied (the interested reader is referred to Zaslavsky's \cite{encyclopedia}). This paper focuses in particular on signed degrees and signed degree sets. The signed degree of a vertex $v$ in a signed graph, denoted by $sdeg(v)$, is the number of positive edges incident to $v$ subtracted by the number of negative edges incident to $v$. For example, Figure $\ref{fig:ex}$ has a  vertex $v_1$ with signed degree -1. Moreover, the signed degree set of a signed graph $(G, s)$ is the set $D = \{sdeg(v) \text{ : }v\in V(G) \}$; equivalently, we say that $(G, s)$ satisfies $D$. The graph in Figure $\ref{fig:ex}$, for example, has signed degree set $\{1, -1 \}$.

\begin{figure}[!h]
    \centering
    \includegraphics[scale = .6]{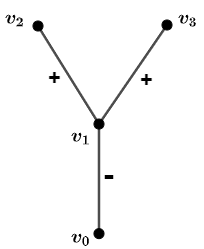}
    \caption{An example of a signed graph}
    \label{fig:ex}
\end{figure}

The topic of signed degree sets stems from the study of signed degree sequences. In \cite{sequence}, Yan \textit{et. al.} found necessary and sufficient conditions needed for an integral sequence to be the signed degree sequence of a signed graph, while noting that a polynomial time algorithm that verifies this condition could be designed. In the same paper, the authors repeat the result but restrict the case to signed trees. In \cite{multipartite}, Pirzada and Naikoo characterized the signed degree sequences of $k-$partite graphs. In the same paper the authors prove that every set of integers is the signed degree set of some $k-$partite graph. Another related subject is studying net regular graphs: those graphs $G$ for which there exists a $\sigma$ such that $(G, \sigma)$ satisfies a set $D$ with $|D|=1$. This concept has been studied \cite{netregular1, netregular2, netregular3} from various perspectives, including its relationship to the spectrum of graphs, regularity degree-wise, and a characterization of net-regular trees. 

In this paper, we investigate, for a given degree set $D$, the minimum diameter that a tree $T$ must have to be the underlying graph of a signed tree $(T,s)$ that satisfies $D$. We also study the same question but in regards to order. We start the paper in Section 2 by determining which sets are valid, i.e. the sets that can be satisfied by a signed tree. In Section 3, we identify the minimum diameter needed when given an arbitrary valid signed degree set. Finally, in Section 4, we repeat this process but for order. Although we have limited results in Section 4, we were able to find the minimum order of signed degree sets that consist of non-negative numbers. In Section 5 we conclude the paper with a conjecture.

Throughout the paper we refer to vertices with degree 1 as pendant vertices. To simplify our notation on signed degree sets, we will always use the variables $x_1, x_2, x_3,...$ and so on to refer to positive integers greater than 1. Similarly, we use the variables $y_1, y_2,y_3, ...$ and so on to refer to negative integers less than $-1$. Finally, we use $z_1, z_2, z_3,...$ and so on to refer to any integer but $1$ and $-1$. 

\section{Valid Sets}
Before we discuss optimization of diameter and order, we must first determine which sets can be considered. An important property of trees is that they have at least two pendant vertices. Since pendant vertices have degree one, they will be incident to either a negative or positive edge, which leads us to the following observation.

\begin{observation}
    If $1\not \in D$ and $-1 \not \in D$, then $D$ cannot be the set of signed degrees of a signed tree.
\end{observation}

\noindent We can take it a step further and consider the following lemma. For it, we will use the notation $CT(n_1, ..., n_m)$, where $n_i>0$, to denote the caterpillar $CT$ where any longest path $P:v_0, ..., v_{m+1}$ has $deg(v_i) = n_i$ for $i\not = 0$ and $i\not = m+1$, and every other vertex in $CT$ has degree 1. We remind the reader that we use $x_i$ to denote positive integers greater than 1, and $y_i$ to denote negative integers less than -1.

\begin{lemma}
    There exists a signed tree $(T, s)$ that realizes $D$ if and only if either $-1\in D$ or $1\in D$.
\end{lemma}
\begin{proof}
Our observation already proves the forward direction. For the other, let $D$ be a set where $-1\in D$ or $1\in D$. We separate the proof into the following cases.

\textbf{Case 1:} $D= \{1, -1, x_1,..., x_n, y_1,..., y_m \}$ and $D'=D\cup \{0 \}$.
    We begin by considering the caterpillar $CT = CT(x_1, ..., x_n, |y_1| + 2, ..., |y_m|)$, where $P = v_0, ..., v_{n+m+1}$ is a longest path. Without loss of generality, let $deg(v_i) = x_i$ for $i\in [n]$, and let $deg(v_{i+n}) = |y_i|+2$ for $i\in [m]$. We will define a function $s: E(CT) \rightarrow \{+, - \}$ such that $(CT, s)$ satisfies $D$. For an edge $e$ in $CT$, define $s$ as follows.
    \[
  s(e) =
  \begin{cases}
                                   + & \text{if $e$ is incident to $v_1, ..., v_n$, or $e$ is of the form $v_iv_{i+1}$,} \\
                                    - & \text{otherwise.}
  \end{cases}
\]
    It is easy to verify then that in $(CT, s)$, for $i\in [n]$, the signed degree of $v_i$ is $x_i$. In a similar way, we can verify that $v_{n+1}, v_{n+2}, ..., v_{m}$ have signed degrees $y_1, y_m$ respectively. In other words, $(CT, s)$ satisfies $D$. The case for $D'$ follows easily.

\textbf{Case 2:} $D= \{1, x_1,...,x_n, y_1,..., y_m \}$ and $D' = D \cup \{0 \}$.
    Let $(T, s)$ be a signed tree that realizes $D\cup \{-1 \}$ like the one described in Case 1, and let $T'$ be that tree where we append two vertices to any pendant vertex of $T$ that has signed degree $-1$. Define $s'$ as a function that extends $s$ to the new edges, assigning $+$ to all of them. It follows that $(T', s')$ realizes $D.$ The case for $D'$ follows easily.

Since proving that there exists a tree that realizes $D$ also proves that there exists a tree that realizes $\{-n \text{ : }n\in D \}$, any case where $-1\in D$ but $1\in D$ can be demonstrated from one of the cases above.
\end{proof}

\noindent As a consequence of this lemma we know which degree sets are of interest for us: the ones with $-1$ or $1$. We thus define these sets as \textbf{valid sets}.  Having established this fact, we are able to delve into optimization.

\section{Minimum Diameter for a Tree}

There are infinitely many signed trees that satisfy a given valid degree set $D$. From these signed trees, we start by studying the ones with optimal diameter.

\begin{definition}
Let $D$ be a valid set. Define its diameter, denoted by $diam(D)$, as the smallest diameter that the underlying tree of a signed tree must have so it satisfies $D$, or equivalently:
$$
diam(D) = \min\{diam(T) \text{ : the signed tree } (T, s) \text{ realizes }D \}.
$$
Further, a signed tree $(T, s)$ is optimal if $diam(T) = diam(D)$.
\end{definition}

\noindent We invite the reader to verify that $diam(\{1, 0 \} ) =3$ and that $diam(\{1, x \}) = 2$ for $x>1$. We will also leave the following without a proof.

\begin{lemma}
For every $k \geq diam(D)$, there exists a signed tree $(T, s)$ that realizes $D$ and whose underlying tree has diameter $k$.
\end{lemma}

\noindent We start by considering degree sets that are composed of only positive values. We remind the reader that throughout the paper we assume $x_i > 1$.

\begin{theorem} \label{diamPositive}
If $D = \{1, x_1,..., x_n \}$, then
\[
  diam(D) =
  \begin{cases}
                                   2 & \text{if $n=1$} \\
                                   3 & \text{if $n=2$} \\
                                    4 & \text{if $n>2$}
  \end{cases}
\]
\end{theorem}

\begin{proof}
Throughout this proof, always assume $s$ maps every edge to a $+$. When $n = 1$, the signed graph $(ST(x_1), s)$ realizes $D$ and has diameter 2, and evidently no signed tree with diameter 1 can realize $D$. For $n=2$, attach $x_2 - 1$ vertices to a pendant vertex in $ST(x_1)$. The resulting signed tree will realize $D$ and have diameter 3. Diameter 2 for this set is not possible because trees of diameter $2$ are stars, and stars only have one vertex with degree higher than 1. Similarly, trees with diameter 3 have exactly 2 vertices whose degree is not equal to 1, so no signed tree of diameter 3 can satisfy a set $D$ when $n > 2$. In other words, when $n > 2$, it must be that $diam(D)\geq 4$. 
It remains to show that $diam(D) \leq 4$. Assume that $x_1 < x_2 < ... < x_n$, and notice that since $x_1>1$, it must be that $x_n > n$. Starting with $ST(x_n)$, attach $x_i$ vertices to a pendant vertex of $ST(x_n)$ for each $i\in [n-1]$. Since $x_n > n$, there will be enough pendant vertices of $ST(x_n)$  to do this. Call this resulting tree $T$. Observe that $(T, s)$ satisfies $D$ and has diameter 4, thus proving that $diam(D) \leq 4$.
\end{proof}

\begin{figure}[h]
    \centering
    \includegraphics[scale = 0.6]{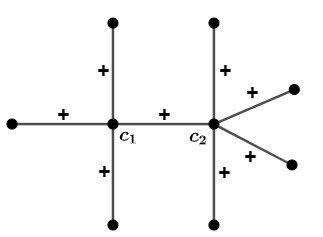}
    \caption{An optimal signed tree by Theorem \ref{diamPositive} given the set D = \{1,4,5\}}
    \label{fig:diam_pos}
\end{figure}

\noindent Building off this theorem, we consider what changes must be made to a tree described in Theorem \ref{diamPositive} (and that can be seen in Figure $\ref{fig:diam_pos}$) in order for it to accommodate a vertex of signed degree $0$. As previously noted, we know that $diam(\{1, 0\}) = 3$, which we now know is higher than the diameter of $\{1, x \}$ for $x>1$. So, when adding additional positive values to the degree set of $\{1, 0\}$, the behavior of optimal diameter closely matches that of the previous theorem. In particular, the diameter of $\{1, 0, x_1, ..., x_{n} \}$ is the same as the diameter of $\{1, x_1, ..., x_{n+1}\}$ when $n \geq 1$. We invite the reader to verify the following.

\begin{corollary}
If $D = \{1, 0, x_1,..., x_n \}$, then
\[
  diam(D) =
  \begin{cases}
                                   3 & \text{if $n=1$} \\
                                   4 & \text{if $n>1$}
  \end{cases}
\]
\end{corollary}

\noindent This technique of using simpler degree sets to prove results of more intricate ones will repeat many times throughout the paper. The following observation is important as it allows us to connect degree sets to bigger ones.

\begin{observation} \label{subsetDiam}
If $A \subset B$, $1\in A$ and $-1\not\in B$, then $diam(A) \leq diam(B)$.
\end{observation}

\begin{figure}[h]
\centering
\includegraphics[scale = .8]{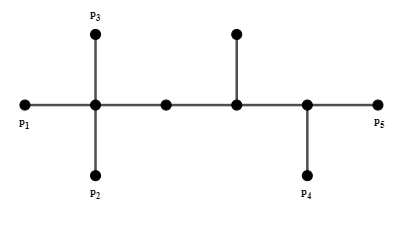}
\caption{A tree with limiting pendant vertices labelled $p_i$, $1 \leq i \leq 5$.}
\label{fig:limPendant}
\end{figure}

\noindent Moving to negative values, we want to prove in general that if $-1\not\in D$ and if $(T, s)$ realizes $D$, then vertices far away from the center of $T$ cannot have negative signed degrees. The following definition helps us formalize this notion. 

\begin{definition}
A vertex $v$ in a tree $T$ is called a limiting pendant vertex if it is a pendant vertex in a longest path of the tree (see Figure \ref{fig:limPendant}).
\end{definition}

\begin{lemma} \label{mNeg}
Let $(T, s)$ be a signed tree that satisfies a set $D$ where $-1\not \in D$. If $v$ is a vertex in $T$ with negative signed degree, then $v$ is not a pendant vertex nor adjacent to a limiting pendant vertex.
\end{lemma}
\begin{proof}
We proceed by contradiction. Let $v$ be a vertex such that $v$ is adjacent to limiting pendant vertex $p$ while $sdeg(v) = y < 0$. Let $S$ be a longest path in $T$ containing $vp$. To satisfy $y$, the vertex $v$ must be incident to at least $|y| + 1$ negative edges. Note that $S$ may contain at most $1$ of these negative edges, so there exists a vertex $w$ such that $s(vw) = -$ and $w$ is not a vertex in $S$. Since $sdeg(w)\in D$ and $sdeg(w) \not = -1$, there must exist some number of additional edges other than $vw$. This, however, allows us to create a longer path than $S$ by using the path that contains $vw$ rather than $vp$, leading to a contradiction with the maximality of $S$. 
\end{proof}

\noindent The first case of degree sets $D$ with negative values we will consider is the one where $1$ is the only positive value in $D$. A signed tree that satisfies such a degree set must have pendant vertices with positive signed degree. As the following result confirms, this implies that we need a larger diameter. We remind the reader that throughout the paper we assume $y_i < -1$.

\begin{lemma} \label{limitingPendant}
If $(T, s)$ is a signed tree that satisfies $D = \{1, y_1, ..., y_m \}$, then 
\[
  diam(T) \geq
  \begin{cases}
                                   4 & \text{if $m=1$.} \\
                                   5 & \text{if $m=2$.} \\
                                   6 &
                      \text{if $m > 2$}.
  \end{cases}
\]
\end{lemma}
\begin{proof}
In a tree of diameter $3$, there exists no vertex that is neither a pendant vertex, nor adjacent to a limiting pendant vertex. By Lemma \ref{mNeg}, a signed tree of diameter $3$ that satisfies $D$ cannot  have a vertex of negative signed degree, so if $m = 1$, then $diam(T, s) \geq 4$. A similar logic can be used to prove the cases where $m = 2$ and $m > 2$.
\end{proof}

\noindent Our penultimate case is the one where $D$ may contain any integer with the exception of $-1$. As the proof of Theorem \ref{anythingbut-1} demonstrates, the fact that negatives require such a big diameter allows us to add positive numbers to the signed degree set without changing this parameter.

\begin{figure}[h]
\centering
\includegraphics[scale = .4]{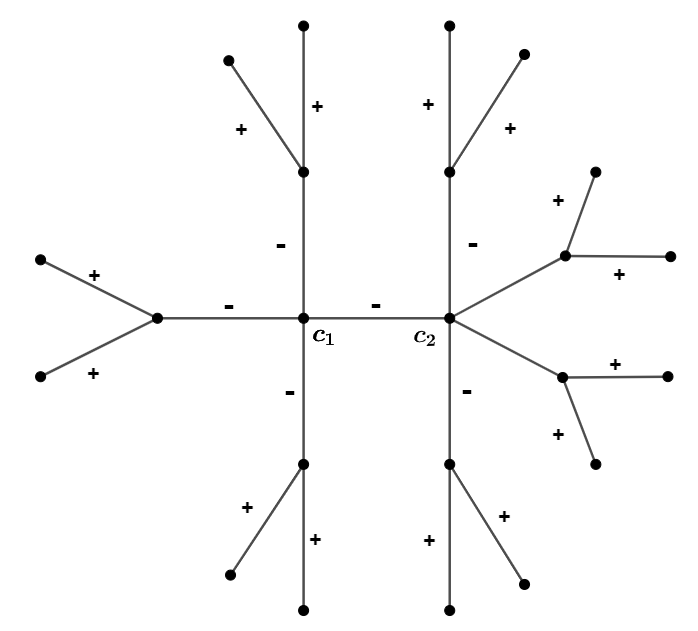}
\caption{An optimal signed tree that satisfies the degree set $D = \{1, -4, -5\}$ constructed from the tree in Figure 2.}
\label{fig:posNegTree}
\end{figure}

\begin{theorem} \label{anythingbut-1}

If $D= \{1, y_1,..., y_m, x_1,..., x_n\}$ where $m \geq 1$ and $n\geq 0$, then $S(D) = S(D\cup \{0 \})$ and
\[
  diam(D) =
  \begin{cases}
                                    4 & \text{if $m=1$} \\
                                    5 & \text{if $m=2$} \\
                                    6 & \text{if $m>2$.}
  \end{cases}
\]
\end{theorem}
\begin{proof}
Consider the case when $m = 1$. Let $D' = \{-1, y_1\}$. Let $(T, s)$ be the signed tree where $T = CT( |y_1|)$ and where $s$ assigns a $-$ to every edge. Note that $(T,s)$ satisfies $D'$. Let $c$ refer to the vertex in $T$ such that $deg(c) = |y_1|$. 
Let $P = \{p_1,..., p_{|y_1|}\}$ be the set of pendant vertices in $T$.
For each $i\in \{1, ..., |y_1| \}$, attach two vertices to vertex $p_i$, and extend $s$ so it assigns $+$ to these new edges. This will change every vertex with signed degree of $-1$ to signed degree of 1, thus making $(T, s)$ now realize $\{1, y_1\}$. For an example of this construction, but applied to the case when $m = 2$, see Figure \ref{fig:posNegTree}. Note that this process will increase the diameter of $T$ to 4. Finally, to make the tree satisfy $D$, we will use a procedure that we first illustrate with $x_1$. Attach two new vertices $u_1$ and $u_2$ to $c$ such that $s(u_1c) = -$ and $s(u_2c) = +$. This way the signed degree of $c$ remains unchanged. Attach $x_1+1$ vertices to $u_1$, and let $s$ assign $+$ to these new edges. This makes the signed degree of $u_1$ be $x_1$. In other words, $(T, s)$ now satisfies $\{1, x_1, y_1 \}$. Observe that the diameter of $T$ will not increase using this procedure. Continuing in this manner for every $x_i \in D$, we see that we can modify $(T, s)$ so that it satisfies $D$ while still having diameter 4. This proves that $diam(D) \leq 4$, and by Lemma \ref{limitingPendant}, we have the result for $m = 1$.

\noindent We can use the same technique when $m = 2$. We start with a signed tree $(T, s)$ as described in Theorem \ref{diamPositive} that satisfies $D' = \{-1, y_1, y_2 \}$, and we modify it so it satisfies $\{1, y_1, y_2 \}$ with diameter 5. Further, since the diameter is large, we can modify the tree so it has a vertex with signed degree $x_i$ while maintaining the diameter. This gives the desired upper bound while Lemma \ref{limitingPendant} gives the desired lower bound. The same technique also works when $m > 2$.
\end{proof}

\begin{figure}[h]
\centering
\includegraphics[scale = .5]{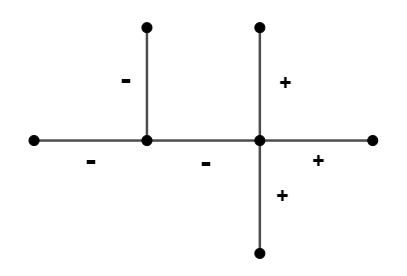}
\includegraphics[scale = .5]{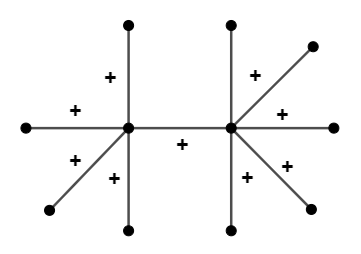}

\caption{The signed graph on the left satisfying $D = \{1, -1, -2, 2 \}$, and the signed graph on the right satisfying $D = \{1, 5, 6 \}$}.
\label{fig:posNegTree2}
\end{figure}

\noindent Our final and most general case for the optimization of diameter is signed degree sets that contain 1 and $-1$ along with $z_1, ..., z_n$, where the numbers $z_i$ are meant to denote distinct integer not equal to 1 or -1. The crux of this proof lies in the fact that the result matches that of Theorem \ref{diamPositive}.
In short, these results are similar because it is possible to build an optimal tree that satisfies a degree set $\{1, -1, z_1, ..., z_n\}$ by modifying an optimal tree that satisfies $\{1, x_1, ..., x_n\}$.

\begin{theorem}\label{generalCase}
If $D = \{1, -1, z_1,..., z_n \}$, then
\[
  diam(D) =
  \begin{cases}
                                   2 & \text{if $n=1$} \\
                                   3 & \text{if $n=2$} \\
                                    4 & \text{if $n>2$.}
  \end{cases}
\]
\end{theorem}
\begin{proof}
Let $(T, s)$ be an optimal signed tree that satisfies $D$. Further, set $\Delta = \Delta(T)$, the maximum degree of $T$, and let $D' = \{1, \Delta + 1, ..., \Delta + n\}$. We will show that $diam(D) = diam(D')$, which by Theorem \ref{diamPositive} will give us the result.

\noindent To prove that $diam(D') \leq diam(D)$, we will show that we can create a tree $(T', s')$ that satisfies $D'$ from $(T, s)$ such that $diam(T) = diam(T')$. This would suffice because of the fact that $diam(D') \leq diam(T') = diam(T) = diam(D)$. To construct $(T', s')$, we start by setting $F$ as a tree such that $F\cong T$. Let $V = \{v_1, ..., v_n\}$ be a collection of non-pendant vertices in $F$. Define $F'$ as that tree that outcomes from attaching $\Delta + i - deg_F(v_i)$ new vertices to each vertex $v_i$. Finally, let $T'$ be that tree that outcomes from attaching to every non-pendant vertex $w \notin V$ in $F'$ the quantity of $\Delta+1-deg_{F'}(w)$ new vertices. Notice that every nonpendant vertex of $T$ has been changed when compared to $T'$. Moreover, for every valid $i$, note that $deg_{T'}(v_i) = \Delta + i$, and for every non-pendant vertex $w \notin V$ we have $deg_{T'} (w) = \Delta + 1$. Thus, if $s'$ assigns $+$ to every edge in $T$, then $(T', s')$ satisfies $D'$. Since we added vertices only to non-pendant vertices, it must be that $diam(T) = diam(T').$ For an illustration of a signed graph $(T, s)$ and $(T', s')$, see Figure 5.

\noindent To prove that $diam(D) \leq diam(D')$, we start by letting $(H, s)$ be an optimal signed tree that satisfies $D'$. We will show that we can create a signed tree $(H', s')$ that satisfies $D$ such that $diam(H') = diam(H)$. Let $p_i = \Delta + i - z_i$, and notice that $p_i > 0$. Set $V = \{v_1, ..., v_n \}$ as a set of vertices in $H$ such that $sdeg(v_i) = \Delta + i$, and let $G$ be that tree that outcomes from attaching $p_i$ new vertices to $v_i$. Finally, for each non-pendant vertex $w$ in $H$ such that $w\not \in V$, let $H'$ be that tree that outcomes from attaching $|sdeg_H(w)| + 1$ vertices [to $w$]. 
It remains to define $s'$. Set $s'$ as the extension of $s$ where edges in $E(G) \setminus E(H)$ get assigned a $-$, and edges in $E(H') \setminus E(G)$ get assigned a $+$ if $sdeg_H(w) < 0$ or a $-$ if $sdeg_H(w) > 0$. Notice then that $sdeg_{H'}(v_i) = \Delta + i - (\Delta + i - z_i) = z_i$, and for nonpendant vertices $w \not = v_i$ we have that $sdeg_{H'}(w) = 1$ because of our choice of the sign that $s$ assigns to the new edges incident to $w$ in $H'$. Every other vertex will be pendant vertices because they were either pendant vertices in $H$, or they were vertices added in the construction, and since we did not add vertices to new ones, they have remained pendant. Thus, $(H', s')$ satisfies $D$. Since we only attached vertices to non-pendant vertices, $diam(H') = diam(H)$. This concludes the proof as $diam (D) \leq diam(H') = diam(H) =  diam(D')$.
\end{proof}

\section{Minimum Order for a Tree}

We now shift our attention to optimizing order, where the order of a graph $G$ is denoted by $\sigma(G)$. Similar to the previous section, we want to optimize order while satisfying a given degree set $D$.

\begin{definition}
Let $D$ be a valid set. Define its order, denoted by $\sigma(D)$, as the smallest order that the underlying tree of a signed tree must have so that it satisfies $D$, or equivalently:
$$
\sigma(D) = \min\{\sigma(T) \text{ : the signed tree } (T, s) \text{  realizes }D \}.
$$
Further, a signed tree $(T, s)$ is now optimal when $\sigma(T) = \sigma(D)$ (so it no longer refers to diameter).
\end{definition}

Unlike our diameter section, we have not found the minimum order for every possible valid set $D$. We have solved only two cases, the first case encompassing only positive integers, and the second case of positive integers and zero. The former case, as with diameter, is the simplest one. We remind the reader that throughout the paper we assume $x_i > 1$.

\begin{theorem} \label{order1}
If $D = \{1, x_1,..,. x_n \}$, then $\sigma(D) = 2 - n + \sum x_i$.
\end{theorem}
\begin{proof}
We will prove the theorem by induction. For $n = 1$, notice that $(ST(x_1), s)$, where s assigns positives to every edge, is a tree of least order that realizes $D$, and that $\sigma(ST(x_1)) = x_1 + 1$ which satisfies the formula. Assume that the result holds for $n - 1$. Let $(T, s)$ be a tree that realizes $D$ with optimal order. Further, let $p$ be a limiting pendant vertex, and set $v$ as the vertex adjacent to $p$. Since every signed degree is positive and the tree is optimal, then every edge must be positive. It follows that $deg(v) = x_a$ for some $a$. If we let $T'$ be that tree obtained by removing every pendant vertex adjacent to $v$, then $T'$ stays connected and $v$ becomes a pendant vertex in $T'$ since $p$ was a limiting pendant. Further, it cannot be that $(T', s|_{E(T')})$ realizes $D$ since $\sigma(T') < \sigma(T)$, so it must be that $(T', s)$ realizes $D- \{x_a \}$. By the induction hypothesis, $\sigma(T') = 2 - (n-1) + \sum x_i - x_a$. This, combined with the fact that $\sigma(T') = \sigma(T) - (x_a - 1)$, completes the inductive step. 
\end{proof}

\noindent Even though $\sigma(D)$ for positives and zero has a very similar expression to that of just positives (see Theorem \ref{last}), the proof of Theorem \ref{last} is much harder. We first need to establish two results about the structure of every optimal tree that satisfies a set with positive values and 0.

\noindent The proof technique we will use for these two results have the following structure: to prove that every optimal signed tree of a set $D$ satisfies a statement $P$, we will demonstrate that if there exists an optimal signed tree $(T, s)$ that fails $P$, then there exists a signed tree $(T', s')$ that satisfies $D$ and has the property that $\sigma(T') < \sigma (T)$. This will be enough as it contradicts the fact $(T, s)$ is optimal.
Further, the underlying tree $T'$ will be based on $T$, having changes made through the "transfer" of vertices. We formalize this in the following definition.

\begin{definition}
Let $T$ be a tree with the distinct vertices $u, v$, and $w$. If $uv\in E(T)$ and the unique path between $u$ and $w$ includes $v$, then the graph $T'$ that outcomes from transferring $u$ to $w$ is the graph with the following properties.
\begin{itemize}
    \item $V(T') = V(T)$, and
    
    \item $E(T') = E(T) \setminus \{uv \} \cup \{uw\}$.
\end{itemize}
\end{definition}

\noindent Notice that the condition of having $v$ be in the path between $u$ and $w$ guarantees that $T'$ is a tree as well. 
\textbf{Since multiple transfers usually occur in a single proof, we will keep denoting the tree after the transfer also as $T$ for the sake of simplicity.}
The proofs and cases are short enough that hopefully the abuse of notation will not cause confusion. In this paper, we use $N(v)$ to refer to the neighborhood of a vertex $v$.

\begin{lemma} \label{lastresult0}
If $(T, s)$ is an optimal signed tree that realizes $D = \{1, 0, x_1,..., x_n \}$, and $ab$ is an edge in $T$ such that $s(ab) = -$, then $sdeg(a) = sdeg(b) = 0$. 
\end{lemma}
\begin{proof}
Assume for a contradiction that $a$ has signed degree $k > 0$. It follows that $a$ there exist $k + 1$ vertices $w_1, ..., w_{k+1}$ adjacent to $a$ such that $s(aw_i) = +$. Transfer the vertices in the set $N(w_1)- \{a \}$ to a pendant vertex $p$ that satisfies the conditions for transfers (i.e. the transfer will not produce a cycle), and modify $s$ so it assign + to these new edges. Notice that the signed degree set of $(T, s)$ will not have changed. If we repeat this process again with $w_2$, $w_3, ...,w_{k-1}$, and $w_{k}$ such that when dealing with $w_j$ we have $p \not = w_i$ with $i < j$, then $T$ becomes a tree where $w_1, ..., w_k$ are pendant vertices. Further, $(T, s)$ will still satisfy the degree set $D$. Finally, if we transfer $w_1, ..., w_{k-1}$ of the vertices to a pendant vertex $p'\not = w_k$ and then delete $w_k$, we notice that $sdeg(p') = k$ and that $sdeg(a) = 0$. Thus, $T$ still satisfies $D$ but its order has decreased, contradicting the assumption that $T$ was optimal at the beginning of the proof. We conclude that $sdeg(a) = 0$. The same argument can be applied to $b$.
\end{proof}

From now on, whenever a transfer happens in a signed tree, we also modify $s$ so it maintains the sign in the new edge unless otherwise stated.

\begin{lemma} \label{lastresult1}
If $(T, s)$ is an optimal signed tree that realizes $D = \{1, 0, x_1,..., x_n \}$, then there is only one edge $e$ in $T$ such that $s(e) = -$.
\end{lemma}
\begin{proof}
For a contradiction, assume that there exists an optimal tree $T$ with at least 2 negative edges: $u_1v_1$ and $u_2v_2$, where the vertices are labeled such that the unique path from $u_1$ to $u_2$ includes both negative edges. By Lemma \ref{lastresult0}, $u_1$ has signed degree 0, so there must exist a vertex $w$ adjacent to $u_1$ such that $u_1w$ is positive. Transfer every vertex in $N(u_1) - \{v_1, w \}$ to $u_2$. This change will not affect the set that $(T, s)$ satisfies because the signed degrees of $u_1$ and $u_2$ will not change. Transfer the vertices in $N(w) - \{u_1 \}$ to a pendant vertex $p$ making sure $p$ satisfies the conditions for transfer. Again, notice that $T$ still satisfies $D$. However, setting $T' = T- \{u_1, w \}$, this implies that $(T', s|_{V(T')})$ satisfies $D$. This is the case because $v_2$ and $u_2$ have signed degree 0, and by deleting $u_1$ we only change the signed degree of $v_1$ to 1. This contradicts the assumption that $T$ was of optimal order. Thus, there cannot be two negative edges in $T$. 
\end{proof}

The following corollary follows immediately from Lemma \ref{lastresult0} and Lemma \ref{lastresult1}

\begin{corollary} \label{lastresult2}
If $uv$ is the unique negative edge in an optimal signed tree $(T, s)$ that satisfies $D$, then $deg(u) = deg(v) = 2$.
\end{corollary}

We will denote the negative edge, and surrounding vertices, by $G_0$, as illustrated in Figure 6. This allows for a very simple proof of our last result. As with diameter, we note that the result of a previous case facilitates the proof of another case.

\begin{figure}[h]
\centering
\includegraphics[scale=0.5]{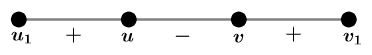} 
\caption{The graph $G_0$.}
\label{g0}
\end{figure}

\begin{theorem} \label{last}
If $D = \{1, 0, x_1,..., x_n \}$, then $\sigma (D) = 4-n + \sum x_i$.
\end{theorem}
\begin{proof}
Let $T$ be an optimal tree that satisfies $D$. As noted, the graph $G_0$ is an induced subgraph of $T$, and $s$ assigns the corresponding signs as indicated in Figure \ref{g0}. There are three cases:

\textbf{Case 1: $deg(u_1) = 1$.} Consider the graph $T' = T - \{u_1, u\}$. $T'$ realizes $D - \{0 \}$. Further, $T'$ is an optimal tree of $D - \{0 \}$ because if it was not, then $T$ would not be an optimal tree for $D$. Thus, $\sigma(T') = \sigma(D - \{0 \}) = 2-n+ \sum x_i$. Since $\sigma(T) = \sigma(T') + 2$, the result holds.

\textbf{Case 2: $deg(u_2) = 1$}. We can apply the same argument as in Case 1.

\textbf{Case 3:} $deg(u_1) \not = 1$ and $deg(u_2) \not = 1$. Let $T_1$ and $T_2$ be the two components of $T- uv$, and let $U$ and $V$ be the set of signed degrees that $(T_1, s|_{E(T_1)})$ and $(T_2, s|_{E(T_2)})$ satisfy respectively. Notice that $U\cup V = D - \{0 \}$, and that $U- \{1 \}$ and $V-\{1 \}$ are mutually exclusive since $T$ is optimal. Similarly, $(T_1, s|_{E(T_1)})$ and $(T_2, s|_{E(T_2)})$ must also be optimal for $U$ and $V$. Thus, by Theorem \ref{order1}, 
\begin{align*}
\sigma(T) & = \sigma(T_1) + \sigma(T_2) \\
          & = 2-(|U|-1)+\sum_{x_i\in U}x_i + 2 - (|V|-1) + \sum_{x_j\in V}x_j \\
          & = 6 - (|U| + |V|) + \sum x_i \\
          & = 6 - (n+2) + \sum x_i \\
          & = 4 - n + \sum x_i.
\end{align*}
\end{proof}

\section{Conclusion}
We have studied two ways to optimize a signed tree, by diameter and by order. In this paper, we found the optimal signed trees by diameter for any given valid degree set. In addition, we found the optimal signed trees in regards to order for when the degree set contains positives and when the degree set contains 0 and positives.

Future research directions include optimizing order for degree sets containing negative values. Observe that the trees constructed in this paper to prove results suggest that optimizing diameter requires a large order, and similarly optimizing order requires a large diameter. We conjecture that if $|D| > 2$ and $(T, s)$ satisfies $D$, then it is not possible to have both $diam(T) = diam(D)$ and $\sigma(T) = \sigma(D)$.

\end{document}